\documentclass[11pt, a4paper]{amsart}

\usepackage{dsfont}
\usepackage{amsmath, amsthm, amssymb}
\usepackage{fullpage}
\usepackage{xypic}
\usepackage{graphicx}
\usepackage{color}
\usepackage[usenames,dvipsnames]{xcolor}
\usepackage[latin1]{inputenc}
\usepackage{indentfirst}
\usepackage{epstopdf}
\usepackage{subfigure}
\usepackage{emptypage}
\usepackage{hyperref}
\hypersetup{
    colorlinks=true, 
    linktoc=all,     
    linkcolor=blue,  
    citecolor=JungleGreen,
}

\newtheorem{theorem}{Theorem}
\newtheorem{corollary}[theorem]{Corollary}

\newtheorem{proposition}[theorem]{Proposition}
\newtheorem{lemma}[theorem]{Lemma}
\newtheorem{definition}[theorem]{Definition}
\newtheorem{remark}[theorem]{Remark}

\newtheorem{theorem*}{Theorem}
\newtheorem{question*}[theorem*]{Question}
\newtheorem{conjecture*}[theorem*]{Conjecture}
\newtheorem{corollary*}[theorem*]{Corollary}
\newtheorem{theorem*e}{Teorema}
\newtheorem{question*e}[theorem*e]{Pregunta}
\newtheorem{conjecture*e}[theorem*e]{Conjetura}
\newtheorem{corollary*e}[theorem*e]{Corolario}

\newcommand{\Fix}{\mathrm{Fix}}

\renewcommand{\int}{\mathrm{int}}
\newcommand{\Inv}{\mathrm{Inv}}
\newcommand{\sgn}{\mathrm{sgn}}
\renewcommand{\deg}{\mathrm{deg}}

\newcommand{\guion}{\,-\,}

\title{Fixed point indices of planar continuous maps}

\author{Luis Hernández--Corbato}
\address{IMPA, Estrada dona Castorina 110, Rio de Janeiro, Brazil.}
\email{luishcorbato@mat.ucm.es}

\author{Francisco R. Ruiz del Portal}
\address{Facultad de Matemáticas UCM, Plaza de Ciencias 3, Madrid, Spain.}
\email{rrportal@ucm.es}

\begin{document}

\begin{abstract}
We characterize the sequences of fixed point indices $\{i(f^n, p)\}_{n\ge 1}$ of fixed points
that are isolated as an invariant set and continuous maps in the plane. In particular, we prove
that the sequence is periodic and $i(f^n, p) \le 1$ for every $n \ge 0$. This characterization allows us to
compute effectively the Lefschetz zeta functions for a wide class of continuous maps in the \(2\)-sphere, to obtain
new results of existence of infinite periodic orbits inspired on previous articles of J. Franks and to give a partial answer to a problem of Shub about the growth of the number of periodic orbits of degree--\(d\) maps in the 2-sphere.
\end{abstract}

\maketitle

\section{introduction}
The fixed point index of a map is a topological invariant which counts with multiplicity the fixed points
of a map. Lefschetz Theorem relates the index to the homological behavior of a map and has been
extensively used to prove existence of periodic points of a map in the literature. The study of the
integer sequence $\{i(f^n, p)\}_{n \ge 1}$ provides a better understanding of the set of periodic points.
Dold proved in \cite{dold} that the sequence of fixed point indices must satisfy a certain modular relations now widely named \emph{Dold's congruences}.
This is the only general result concerning these sequences.

Probably the best well--known restriction appears in the $C^1$ category. Shub and Sullivan proved in \cite{shubsullivan} that
the fixed point index sequence $\{i(f^n, p)\}_{n \ge 1}$ is periodic if $f$ is $C^1$.
In the low--dimensional setting, topology plays a major role. Periodicity holds
as long as we consider planar homeomorphisms, see \cite{brown, lecalvezENS}.
Yet another hypothesis which also gives rise to major obstructions is to consider fixed points isolated as invariant sets as in Conley index theory, in other words, assume that the fixed point is the maximal invariant set of any of its neighborhoods.
This extra assumption enforces a very rigid behavior of the sequence of indices for planar homeomorphisms
(\cite{lecalvezyoccoz,lecalvezyoccozconley, rportalsalazar})
and also enforces periodicity in dimension 3 (\cite{london}), where some additional restrictions appear in the orientation--reversing case (\cite{gt}).

The objective of this work is to characterize the fixed point index sequence for
planar continuous maps $f$ and fixed points $p$ isolated as invariant sets.
Our result represents, presumably, the only
non--trivial existing restriction in the continuous case.
Babenko and Bogatyi in \cite{babenko} and then Graff and Nowak--Przygodzki in \cite{graff}
showed that there is no restriction other than Dold's congruences in the planar continuous case
as long as no hypothesis at all is made about the kind of isolation of the fixed point.

The main theorem in this article, which solves a problem posed by Graff, is the following (notation will be explained in Subsection \ref{sec:index}).

\begin{theorem}\label{thm:formula}
Let $U$ be an open subset of the plane, $f : U \to f(U) \subset \mathds{R}^2$ be a continuous map
and $p$ a fixed point of $f$ which is isolated as an invariant set and is neither a source nor a sink. Then,
$$\{i(f^n, p)\}_{n \ge 1} = \sigma^1 -  \sum_{k \in F} a_k \sigma^k$$
where $F$ is a finite non--empty subset of $\mathds{N}$ and $a_k$ is a positive integer for every $k \in F$.
\end{theorem}

From our theorem we easily recover the main result of \cite{graffpaco}.

\begin{corollary}\label{cor:index<1}
With the previous hypothesis, the sequence $\{i(f^n, p)\}_{n\ge 1}$ is periodic and bounded from above by 1.
\end{corollary}

The proof of the characterization theorem involves overcoming some technical issues.
Our approach also relies heavily in Conley index theory, but
only uses its basic tools, making our exposition mostly self--contained.
We follow the approach to the discrete Conley index explained in \cite{franksricheson},
another standard introductory reference is \cite{handbookconley}. The application of this theory
to our setting, the study of indices of locally maximal fixed points,
started with Franks in \cite{franksminimal} and since then has proved to be successful in the proofs
of the aforementioned results. More precisely, the idea behind the arguments is the description of the first
homological discrete Conley index, for a detailed and enlightening discussion in the case of homeomorphisms see \cite{gt}.

We will also deal with a particular case of an old problem posed by Shub.
Let \(S^2\)  be the 2-sphere, oriented in the standard way. Fix a continuous map \(f : S^2 \rightarrow S^2\) of global degree
\(2\). That is, the map \(f_{*,2} : H_2(S^2) \rightarrow H_2(S^2)\) is multiplication by \(2\). Problem 3 of \cite{shub} asks: Is the Growth Rate Inequality
\begin{equation}\label{eq:growth}
\limsup_{n \to \infty} \frac{1}{n}\ln N_n(f) \geq \ln(2)
\end{equation}
true?
Here \(N_n(f)\) denotes the number of fixed points of $f^n$. 
A very simple map whose dynamics is North Pole to South Pole provides a negative answer in the $C^0$ category.
Recently, the inequality has been shown to be true in some special cases: for $C^1$ latitude--preserving maps
in \cite{pughshub} and for maps with two attracting fixed points in \cite{uruguayos}.
It is conjectured to hold as well for any $C^1$ map but, up to the authors knowledge, no further
significant progresses have been made towards the solution.
We will show that the answer is also positive for maps such that all its periodic orbits are isolated as invariant sets provided 
$f$ has no sources of degree $r$ with $|r| \ge 1$.

\subsection*{Acknowledgements} The authors would like to thank our collaborator Patrice Le Calvez
for sharing his insightful ideas during the
preparation of the article \cite{gt} which have proved to be equally useful to conduct this research.
The authors have been supported by MINECO, MTM2012-30719, and the second author also by a P\'os--Doutorado Junior
grant (162586/2013-2) from CNPq.

\section{Definitions and basic concepts}

\subsection{Planar topology}
Let us fix notations. If $A \subset B$ are subsets of a topological space we will denote
$\int(A), \overline{A}, \partial A, \partial_B A$ and $\pi_0(A)$ the interior, adherence, boundary,
boundary relative to $B$ and set of connected components of $A$. The relative boundary of $A$ in $N$ is defined as
$\partial_B{A} = A \cap (\overline{B \setminus A})$.

Recall that a \emph{Jordan arc} is the image of a one-to-one continuous map $\phi : [0,1] \to S^2$.
In the case $\phi(0) = \phi(1)$ it is called a \emph{Jordan curve}.
A compact subset $A$ of $S^2$ or $\mathds{R}^2$ will be said to be \emph{regular} if it is a 2-manifold with boundary.
The boundary of a regular set $A$ is a 1-manifold, hence a disjoint union of Jordan curves.

It is interesting to remark that if $\alpha$ is a Jordan arc or a Jordan curve contained in $\partial A$,
the quotient space $\widehat A$ obtained from $A$ by
identifying $\alpha$ to a point is homeomorphic to a 2-submanifold with (possibly empty) boundary of $S^2$.
Clearly, if we identify every Jordan curve in which the boundary of $A$ is decomposed to a different point we obtain
a set homeomorphic to the 2-sphere.
Note also that, for every regular set $A$, the homology group $H_1(A, \mathds{Q})$ is generated
by the 1-cycles represented by the connected components of $\partial A$ when viewed as 1-chains.
As $\mathds{Q}$ will always be the coefficient ring in the homology groups we will omit it from the notation.


A great part of technical issues which have been dealt with in this work concerns
the relative display of regular sets. Given two regular sets $A \subset B$ of $S^2$,
the inclusion induces a map
$$\iota : \pi_0(S^2 \setminus B) \to \pi_0(S^2 \setminus A).$$
Similarly, a homomorphism between reduced 0-homology groups is established. It follows from Alexander duality
that the dual of this homomorphism is the map $H_1(A) \to H_1(B)$ induced by inclusion.
Thus, this map is surjective if and only if $\iota$ is injective.
There is no need to take adherences for $S^2 \setminus A$ and $S^2 \setminus B$ because of the assumption on regularity.

Recall that a compact subset $K$ of the plane is said to be \emph{cellular} if it has a basis of neighborhoods
composed of topological closed disks. This is equivalent, only in dimension 2, to being an \emph{acyclic continua}.
It is well--known that any given open neighborhood $U$ of a cellular set $K$ is homeomorphic to the quotient space $U/K$.
It follows that Theorem \ref{thm:formula} holds if we replace the fixed point $p$ by a cellular invariant set $K$.

\subsection{Fixed point index}\label{sec:index}

Let $U$ be an open subset of $\mathds{R}^d$ and $f : U \to f(U) \subset \mathds{R}^d$ a continuous map.
The index of $f$ at a fixed point $p$ will be denoted $i(f, p)$. The integer sequence $\{i(f^n, p)\}_{n \ge 1}$
will be called \emph{fixed point index sequence} (of $f$ at $p$).

After Dold \cite{dold}, it is known that the fixed point index sequence satisfies non--trivial relations
often denominated Dold's congruences. Define the \emph{normalized sequences} $\sigma^k = \{\sigma^k_n\}_{n \ge 1}$, where
$$
\sigma^k_n =
\begin{cases}
k & \text{if } n \in k \mathds{N} \\
0 & \text{otherwise.}
\end{cases}
$$
One way to express Dold's congruences is that any fixed point index sequence $I = \{I_n\}_{n \ge 1}$
is written (uniquely) as a formal combination of normalized sequences $I = \sum_{k \ge 1} a_k \sigma^k$ where each $a_k$ is an integer.

The following lemma is a classical result in this topic. For the proof
and more information on the subject we refer the reader to \cite{babenko}.

\begin{lemma}\label{lem:index}
Given a fixed point index sequence $I = \{I_n\}_{n \ge 1}$, the following statements are equivalent:
\begin{enumerate}
\item $I$ is bounded.
\item $I$ is periodic.
\item There exists maps $\varphi': J' \to J', \varphi: J \to J$ where $J', J$ are finite sets such that
$$I_n = \#\Fix(\varphi'^n) - \#\Fix(\varphi^n).$$
\item Only a finite number of $a_k$ is non--zero.
\end{enumerate}
\end{lemma}

The conclusion of Theorem \ref{thm:formula} can be reformulated in view of the previous lemma as:

\emph{There exists a map $\varphi : J \to J$ over a finite set $J$ such that for every $n \ge 1$
$$i(f^n, p) = 1 - \#\Fix(\varphi^n).$$}

This language is particularly useful to describe the previous results about homeomorphisms in a concise way.

\begin{theorem}[Le Calvez - Yoccoz \cite{lecalvezyoccoz, lecalvezyoccozconley}, Ruiz del Portal - Salazar
\cite{rportalsalazar}]
Let $U$ be an open subset of the plane, $f : U \to \mathds{R}^2$ be an orientation--preserving
(orientation--reversing) homeomorphism and $p$ a fixed point of $f$ which is isolated as invariant set
and is neither a source nor a sink. Then, there exists an orientation--preserving (orientation--reversing)
homeomorphism $\Phi : S^1 \to S^1$ and a finite non--empty set $J$ invariant under $\Phi$ such that,
if we denote $\varphi = \Phi_{|J}$,
$$i(f^n, p) = 1 - \#\Fix(\varphi^n)$$
for every $n \ge 1$.
\end{theorem}

The formula obtained in Theorem \ref{thm:formula} is clearly the expected characterization in the case of
continuous maps in view of the previous theorem. Indeed, a continuous map in $S^1$ may have an
arbitrary number of periodic orbits with different periods from which $J$ can be made up.

\subsection{Tools from Conley index theory}\label{sec:conley}


In the following, let $M$ be a locally compact metric space, $U \subset M$ open and $f : U \to M$ continuous.
We follow the approach of Franks and Richeson \cite{franksricheson} to the discrete Conley index theory.

\begin{definition}
A set $X \subset M$ such that $f(X) = X$ is called \emph{invariant} (under $f$).
A solution through $x \in M$ is a sequence $\{x_n\}_{-\infty}^{\infty}$ such that $x_0 = x$ and $f(x_{n}) = x_{n+1}$
for every $n \in \mathds{Z}$.
Given any set $N \subset M$, the \emph{maximal invariant subset} of $N$, denoted $\Inv(f, N)$,
is the set of points $x \in N$ for which there exists a solution through $x$ contained in $N$.
\end{definition}

Observe that $f(\Inv(f, N)) = \Inv(f, N)$ but they may not be equal to $f^{-1}(\Inv(f, N))$.
Note that, in general, $\Inv(f, N)$ is different from $\bigcap_{n \in \mathds{Z}}f^{-n}(N)$,
see Easton \cite{easton} for an example.
However, it is not difficult to derive a closed formula for $\Inv(f, N)$
(cf. Proposition 2.2 in \cite{franksricheson}):
\begin{equation}\label{eq:invN}
\Inv(f, N) = \bigcap_{m \ge 0}f^m\left(\bigcap_{n \ge 0} f^{-n}(N)\right).
\end{equation}

\begin{definition}
A compact set $N \subset U$ is called an \emph{isolating neighborhood} for $f$ if
$$\Inv(f, N) \subset \int(N).$$
Conversely, a compact invariant set $X$ is called \emph{isolated} if there is an isolating
neighborhood $N$ such that $X = \Inv(f, N)$.
A compact set $B$ is called an \emph{isolating block} for $f$ provided that it satisfies
$$f^{-1}(B) \cap \partial B \cap f(B) = \emptyset,$$
or, equivalently, there are no interior discrete tangencies, i.e., points $x \in B$ such that
$f(x) \in \partial B$ and $f^2(x) \in B$.
\end{definition}

In particular, any isolating block is an isolating neighborhood.
Given an invariant set $X$ we will say $N$ is an isolating neighborhood \emph{of $X$} or
$B$ is an isolating block \emph{of $X$} if $\Inv(f, N) = X$ or $\Inv(f, B) = X$, respectively.

One of the fundamental results in Conley index theory asserts that any compact isolated invariant set $X$ has a basis
of neighborhoods composed of isolating blocks.

\begin{lemma}\label{lem:isolatingblocks}
Let $B$ be an isolating block. Then
\begin{enumerate}
\item any connected component of $B$,
\item any set sufficiently close to $B$ and
\item $B \cap f^{-1}(B)$
\end{enumerate}
are all also isolating blocks.
\end{lemma}
\begin{proof}
The first two statements are left to the reader. For the third one, denote $A = B \cap f^{-1}(B)$.
Since $A \subset B$ and $B$ is an isolating block, internal discrete tangencies can only hipotetically occur
in $\partial_B A$. However, any $x \in \partial_B A$ satisfies $f(x) \in \partial B$, hence $f^2(x) \notin B$
so $f(x) \notin A$. Therefore $f(\partial_B A) \cap A = \emptyset$ and, as a consequence, $A$ is an isolating block.
\end{proof}

Using the previous lemma we can always assume isolating blocks are regular and connected whenever
the invariant sets are connected.
This is the case in this work, where we will ask for some extra properties to these already ``nice'' isolating neighborhoods
to gain geometric control over their dynamics.

\subsection{Nilpotence}

The indices of fixed points will be computed by applying Lefschetz Theorem
after modifying the original map in the boundary of an isolating block.
The task would be simplified if we are allowed to make any assumption of the topology of the isolating block.
Unfortunately, no result in this direction is available in our setting as this question has only been
addressed for homeomorphisms (see \cite{lecalvezyoccozconley, rportalsalazar}).
The following lemma helps to circumvent this shortage, showing how the homological information is trivialized
by the dynamics in an appropriate setting.
It shares the same taste with the work of Richeson and Wiseman in \cite{richesonwiseman}.

\begin{lemma}\label{lem:nilpotent}
Let $S$ be a topological space, $r \ge 1$ and $g : S \to S$ a continuous map such that:
\begin{itemize}
\item[$(i)$] There exist compacts sets $Y, W \subset S$ and an integer $l \ge 0$ such that $g^l(W) \subset Y$.
\item[$(ii)$] The map induced by inclusion $H_r(Y) \to H_r(S)$ is trivial.
\item[$(iii)$] The map induced by inclusion $H_r(W) \to H_r(S)$ is surjective.
\end{itemize}
Then, the map $g_{*,r} : H_r(S) \to H_r(S)$ is nilpotent.
\end{lemma}
\begin{figure}[htb]
\centering
\includegraphics[scale = 0.7]{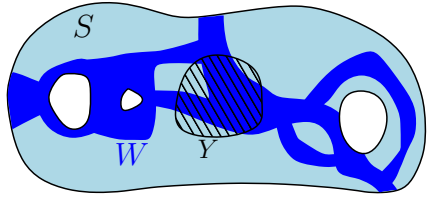}
\caption{Proof of Lemma \ref{lem:nilpotent}: The image of $W$ (in dark) by $f^l$ is contained in $D$ (filled with lines).}
\label{fig:nilpotenciaR2}
\end{figure}
\begin{proof}
We will show that $g^l_{*,r} : H_r(S) \rightarrow H_r(S)$ is the zero morphism.
Fix a class $\alpha \in H_r(S)$, which by $(i)$ is represented by a $r$-chain $\sigma$ contained in $W$.
The homology class $g^l_{*,r}(\alpha)$ is represented by the $r$-chain $g^l(\sigma)$,
which is contained in $g^l(W) \subset Y$ by $(ii)$.
We deduce from $(iii)$ that $g^l(\sigma)$ is a boundary in $S$. Therefore, $g^l_{*,r}(\alpha) = 0$.
\end{proof}

This lemma plays a fundamental role in the proof of Theorem \ref{thm:formula}.
The first hypothesis on the set $W$ will be automatically satisfied if it is a neighborhood
of the stable set of $g$ in $S$, $\Lambda^+(g, S)$, which consists of the set points whose forward orbit under $g$
is contained in $S$. Clearly,
$$\Lambda^+(g, S) = \bigcap_{n \ge 0}g^{-n}(S).$$

\subsection{Attractors}

Recall that $U$ is an open subset of a locally compact metric space $M$ and $f : U \to M$ is continuous.
\begin{definition}
An isolated invariant set $X$ will be called an \emph{attractor} provided there exists an isolating
neighborhood $N$ of $X$ such that
$$f(N) \subset N.$$
If the set $X = \{p\}$, we say that $p$ is an \emph{attracting} fixed point or a \emph{sink}.
If $X$ is the orbit of a periodic point $p$ we say that $p$ is an \emph{attracting} periodic point or a \emph{periodic sink}.
\end{definition}

Equation (\ref{eq:invN}) applied to such $N$ allows to give a simple formula for $X = \Inv(f, N)$:
\begin{equation}\label{eq:invsink}
\Inv(f, N) = \bigcap_{m \ge 0}f^m\left(\bigcap_{n \ge 0}f^{-n}(N)\right) =
\bigcap_{m \ge 0}f^m(\Lambda^+(f, N)) = \bigcap_{m \ge 0} f^m(N),
\end{equation}
where the last equality comes from the fact that $\Lambda^+(f, N) = N$ if $f(N) \subset N$.

\begin{remark}
The forward orbit of any point $x \in \partial N$ is eventually contained in $\int(N)$.
A small argument shows that we can slightly thicken $N$ to obtain an isolating neighborhood $N'$
of $X$ such that $f(N') \subset \int(N')$. Consequently, any attractor has an isolating neighborhood $N'$
such that
$$N' \text{ is regular }\enskip \text{ and } \enskip f(N') \cap \partial N' = \emptyset.$$
In particular, such $N'$ would be an isolating block.
\end{remark}

A direct application of Lemma \ref{lem:nilpotent} allows to compute indices of sinks.

\begin{proposition}\label{prop:sinks}
Assume $M = \mathds{R}^d$, $d \ge 1$.
The fixed point indices of a sink are $i(f^n, p) = 1$ for every $n \ge 1$.
\end{proposition}
\begin{proof}
It suffices to take a regular $N$ as in the definition of sink and $D \subset N$ any topological closed disk
containing $p$ in its interior. From $\Inv(f, N) = \{p\} \subset \int(D)$ and Equation (\ref{eq:invsink})
 we get that $$f^l(N) = \bigcap_{m = 0}^{l}f^m(N) \subset D$$
for sufficiently large $l$. Lefschetz Theorem together with Lemma \ref{lem:nilpotent}
applied to $l$ and $S = W = N$ and $Y = D$ then yield the result.
\end{proof}

\subsection{Repellers}

The concept of repeller for non--invertible maps is not completely clear.
The existence of a backward invariant neighborhood is too restrictive in this setting, as
it no longer shares a local flavour.
Here, we try to keep our discussion as general as possible.

\begin{definition}\label{def:repeller}
An isolated invariant set $X$ will be called a \emph{weakly repeller} if there exists an isolating neighborhood $N$ of $X$
such that
\begin{enumerate}
\item[$(i)$] $f(\partial N) \cap \int(N) = \emptyset$.
\end{enumerate}
If, in addition, it satisfies
\begin{enumerate}
\item[$(ii)$] The forward orbit of every $x \in N \setminus X$ eventually exits $N$.
\end{enumerate}
then $X$ will be called a \emph{repeller}.
If the set $X = \{p\}$ is a repeller, we say that $p$ is a \emph{repelling} fixed point or a \emph{source}.
If $X$ is the periodic orbit of a periodic point $p$,
we say that $p$ is a \emph{repelling} periodic point or a \emph{periodic source}.
\end{definition}


\begin{lemma}
Property $(ii)$ is equivalent to:

\center{For every $x \in \Lambda^+(f, N)$, $f^{-1}(x) \cap N \neq \emptyset$.}
\end{lemma}
\begin{proof}
Clearly, $(ii)$ is equivalent to $X = \Lambda^+(f, N)$ so $(\Rightarrow)$ is obvious.
To prove $(\Leftarrow)$ take $x_0 \in \Lambda^+(f, N)$. By hypothesis, we can define a solution $\{x_n\}_{-\infty}^{\infty}$
contained in $N$ through $x_0$, hence $x_0 \in X$ and only the points of $X$ have their forward orbit contained in $N$.
\end{proof}

%

\begin{remark}\label{rmk:nopreimage}
Notice the dichotomy for a weakly repelling fixed point $p$: either it is repelling or
the isolating neighborhood $N$ in the definition contains a point $x$ which has no preimage in $N$
and whose forward orbit tends to $p$.
\end{remark}

\begin{remark}\label{rmk:regularrepeller}
For weakly repellers the forward orbit of any point in the boundary of $N$ eventually exits $N$.
Therefore, after suitably trimming the boundary of $N$ we obtain an isolating neighborhood $N'$ of $X$
such that $f(\partial N') \cap N' = \emptyset$ (equivalently, $\partial N' \cap f^{-1}(N') = \emptyset$).
As a consequence, in Definition \ref{def:repeller} it is possible to replace $(i)$ by
\begin{center}
$(i')$ $N$ is regular \enskip and \enskip $f(\partial N) \cap N = \emptyset$.
\end{center}
\end{remark}


The computation of fixed point indices of sources seems to be more involved.
Here, we present an argument valid for the planar case.

\begin{proposition}\label{prop:sources}
Let $U$ be an open subset of the plane, $f : U \to f(U) \subset \mathds{R}^2$
and $p$ a source.
Then, there exists an integer $d$ such that $i(f^n, p) = d^n$ for any $n \ge 1$.
\end{proposition}
\begin{proof}
Take a regular neighborhood $N$ of $p$ given in Remark \ref{rmk:regularrepeller}.
Recall that $f(\partial N)$ lies outside $N$.
Now, identify each connected component of $\partial N$ to a different point and denote $\widehat N$
the quotient space.
Note that $\widehat N$ is homeomorphic to the 2-sphere.
Since each $\lambda \in \pi_0(\partial N)$ is a Jordan curve, it bounds an open disk
which does not meet $N$.
The map $f$ naturally induces a map $\widehat f : \widehat N \to \widehat N$ for if $f(x) \notin \int(N)$ then
$\widehat f(x)$ is defined as the point which represents
the connected component of $\partial N$ which bounds the closed disk of $S^2 \setminus \int(N)$
to which $f(x)$ belongs.
Since every closed disk of $S^2 \setminus \int(N)$ is sent inside a unique connected component of $S^2 \setminus N$,
the map $\widehat f$ is well-defined and continuous.
Clearly, $f$ is conjugate to $\widehat f$ in a neighborhood of $p$ and $\widehat p$, the fixed point
of $\widehat f$ corresponding to $p$, respectively. This implies $i(f^n, p) = i(\widehat f^n, \widehat p)$, for every $n \ge 1$.
Moreover, the set $F' = \{q_1,\ldots,q_m\}$ composed of the points which correspond to the connected
components of $\partial N$ is a local attractor. In fact, $\widehat f$
is locally constant in a neighborhood of $F'$.

Let $F = \Inv(\widehat f, F')$ and $k$ such that $\widehat f^k$ is the identity map restricted to $F$.
For any $q_i \in F$, the basin of attraction of $q_i$ for $\widehat f^k$, that is the set of points in $\widehat N$
whose forward orbit under $\widehat f^k$ tends to $q_i$, is an open proper subset of $\widehat N$ by
definition of $F$ and $k$.
From property $(ii)$ of the definition of source we deduce that any point in $\widehat N \setminus \{\widehat p\}$ belongs to
the basin of attraction of some $q_i$ for $\widehat f^k$. Thus, the set of basin of attractions are a partition into
non--empty open subsets of the connected set $\widehat N \setminus \{\widehat p\}$. It then follows that $F$ is composed of only one element, $q$, a sink of $\widehat f$.

Applying Lefschetz Theorem yields that
$$1 + i(f^n, p) = i(\widehat f^n, q) + i(\widehat f^n, \widehat p) = i(\widehat f^n) =
\Lambda(\widehat f^n) = 1 + \deg(\widehat f)^n,$$
where $\deg(\widehat f)$ denotes the degree of the map $\widehat f : \widehat N \to \widehat N$ as a map of the 2-sphere,
and $\Lambda(\widehat f^n)$ denotes the Lefschetz number of $\widehat f^n$.
The conclusion trivially follows.
\end{proof}


\section{Proof of Theorem \ref{thm:formula}: Characterization of fixed point index sequences}

The plan of the proof is as follows. The first goal is to construct a suitable isolating block $B$
 for $\{p\}$ which makes it possible to encapsulate the dynamics, much alike as in Conley index theory.
The definition is given in Step I and some properties are proved in Step II.
 The main difficulty is that in our case
we need to identify each connected component of the exit set of the isolating neighborhood to a different point
so that we do not lose the information on the way any orbit exit $B$.
After the identification we obtain in Step III a space $\widehat B$ equipped with an induced dynamics given by $\widehat f$
in which computations will be carried out. The set of way outs of $B$ and their dynamics
will be then represented by attracting periodic orbits by $\widehat f$ lying in the boundary of $\widehat B$.
The original choice of $B$ proves crucial so that we can apply Lemma \ref{lem:nilpotent} to
$\widehat f$ and $\widehat B$ in Step IV to compute the index.

\medskip
\emph{Step I: Definition of suitable isolating blocks $B \subset B' \subset B_0$.}

Start with a regular connected isolating block $B_0$ of $p$.
By Lemma \ref{lem:isolatingblocks}, the set $A = B_0 \cap f^{-1}(B_0)$ is an isolating block
an so it is the connected component of $A$ that contains $p$, denoted $A_0$.
We have that $f(\partial_{B_0}A_0) \cap A_0 = \emptyset$. Using again Lemma \ref{lem:isolatingblocks},
there exists a regular connected isolating block $B_1$ of $p$ such that
$B_1 \subset A$, $\overline{B_0 \setminus B_1}$ is regular and
$f(\partial_{B_0}B_1) \cap B_1 = \emptyset$. By induction we can find a nested sequence
$\{B_n\}_{n \ge 0}$ which satisfies:
\begin{itemize}
\item $B_{n+1}$ is a regular connected isolating block of $p$ and $\overline{B_n \setminus B_{n+1}}$ is also regular.
\item $f(B_{n+1}) \cup B_{n+1} \subset B_n$.
\item $\partial_{B_n}B_{n+1} = \partial_{B_0}B_{n+1}$ and $f(\partial_{B_0}B_{n+1}) \cap B_{n+1} = \emptyset$.
\end{itemize}


Moreover, $\partial_{B_0}B_n$ is always non--empty for $n \ge 1$. Otherwise, $\partial_{B_0}B_n = \partial_{B_{n-1}}B_n = \emptyset$
so $f(B_{n-1}) \subset B_{n-1}$ and it follows that $p$ is a sink.

Consider, for any $n \ge 1$, the map induced by the inclusion $B_n \subset B_0$
$$\iota_n : \pi_0(S^2 \setminus B_0) \to \pi_0(S^2 \setminus B_n).$$
Since $B_0$ is regular, the set of connected components of its complement is finite and
there exists $m$ so that the cardinality of the image of the maps $\iota_n$ for $n \ge m$
is equal (and minimal).

Denote $B' = B_m, B = B_{m+1}$.

\medskip
\emph{Step II: Properties satisfied by $B, B'$ and $B_0$.}

\begin{lemma}\label{lem:uniqueccR2}
Any connected component $C$ of $\overline{B' \setminus B}$ contains exactly
one connected component of $\partial_{B_0}B$.
\end{lemma}
\begin{proof}
Since $B'$ is connected and contains $B$ it is immediate to check that $\partial_{B_0}B = \partial_{B'}B$ meets $C$.
Suppose there are two different components $\lambda, \lambda' \in \pi_0(\partial_{B_0}B)$ contained in $C$.
Observe that none of them can be a Jordan curve because $B$ is connected. Consequently,
$\lambda$ and $\lambda'$ are Jordan arcs.
Take a path in $B$ joining $\lambda$ and $\lambda'$ and extend it to a Jordan curve $\gamma \subset
B \cup C$, which we can assume only meets once both $\lambda$ and $\lambda'$.
Since any end point of a Jordan arc in $\pi_0(\partial_{B_0}B)$
belongs to the adherence of a connected component of $S^2 \setminus B_0$,
one deduces that there exist connected components $H, H'$ of $S^2 \setminus B_0$
which lie in different components of $S^2 \setminus \gamma$.
Assume each of the adherences of $H$ and $H'$ contains one end point of $\lambda$.
As $\gamma \subset B'$, $H$ and $H'$ are separated by $\gamma$ hence
 lie in different connected components of $S^2 \setminus B'$ and, in particular, $H \neq H'$.
However, it is clear that it is possible to connect $H$ and $H'$ by a path, close to the arc $\lambda$, which does not meet $B$.
Thus, $\iota_m(H) \neq \iota_m(H')$ but $\iota_{m+1}(H) = \iota_{m+1}(H')$, which contradicts the definition of $m$.
\begin{figure}[htb]
\centering
\includegraphics[scale = 0.7]{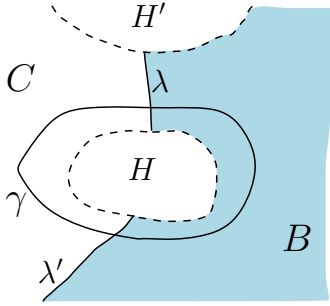}
\caption{Proof of Lemma \ref{lem:uniqueccR2}.}
\label{fig:uniqueccR2}
\end{figure}
\end{proof}

The previous lemma allows to define a map $\Phi : \overline{B' \setminus B} \to \pi_0(\partial_{B_0}B)$
which sends each point $x$ that belongs to a connected component $C$ of $\overline{B' \setminus B}$,
to the unique connected component of $\partial_{B_0}B$ which is contained in $C$.
The existence of $\Phi$ is essential to show that $f$ induces a map $\widehat f$ in the quotient space
$\widehat B$ which will be later constructed.

Next, consider a closed topological disk $D \subset \int(B)$ containing $p$.
A compactness argument shows that $f^l(\Lambda^+(f, B)) \subset \int(D)$.
Then, take $V$ a regular compact neighborhood of the stable set $\Lambda^+(f, B)$ in $B$ so that $f^l(V) \subset D$.
There is a property of $V$ arising from the definition of $m$.
From the definition of the sequence of isolating blocks we obtain that $B_{m+k} \subset V$
for sufficiently large $k\ge 1$. Since the map
$$\iota_V : \pi_0(S^2 \setminus B_0) \to \pi_0(S^2 \setminus V)$$
is then a factor of $\iota_{m+k}$, we have that the cardinality of the images of $\iota_V$,
$\iota_{m+k}$ and hence of $\iota_{m+1}$ are equal.
Rephrasing the previous sentence using duality, the images of the maps
$$H_1(V) \to H_1(B_0) \enskip \enskip \text{and} \enskip \enskip H_1(B) \to H_1(B_0)$$
are equal.

\medskip
\emph{Step III: Encapsulating the local dynamics. Definition of $\widehat B$ and $\widehat f: \widehat B \to \widehat B$.}

Now, we identify each connected component of the relative boundary $\partial_{B_0}B$ to a different point.
Denote the projection map $\pi: B \to \widehat{B}$. 
Observe that there are mainly two cases, $\partial_{B_0}B \neq \partial B$ and $\partial_{B_0}B = \partial B$.
In the first case $\widehat B$ is homeomorphic to a regular subset of the plane whereas in the second
$\widehat B$ is homeomorphic to $S^2$.

Since $V \cap \partial_{B_0}B = \emptyset$, the sets $\widehat{V} := \pi(V)$ and $V$ are homeomorphic. The image of any
$\lambda \in \pi_0(\partial_{B_0}B)$ under $\pi$ is a single point denoted $q_{\lambda}$.

First, we show that the map $H_1(\widehat{V}) \to H_1(\widehat{B})$ is surjective. The group $H_1(B)$
is generated by classes $[\lambda]$ represented by components $\lambda \in \pi_0(\partial B)$
seen as 1--chains. Similarly, $H_1(\widehat{B})$ is generated by classes $[\pi(\lambda)]$ such that
$\lambda \in \pi_0(\partial B)$ and $\lambda$ is not contained in $\partial_{B_0}B$ (otherwise $\pi(\lambda)$
is a point in the interior of $\widehat B$).

\textit{Claim:} The kernel of the inclusion-induced map $H_1(B) \to H_1(B_0)$ is contained
in the subspace $L$ generated by the classes $[\lambda]$ represented by connected components
$\lambda$ of $\partial B$ which are completely contained in $\partial_{B_0}B$.

The claim follows from this fact: given any topological closed disk $E \subset \int(B_0)$ with $\partial E \subset B$,
 $\partial E$ is homologous, as a chain, to the oriented sum of the boundaries of the disks
$E \setminus B$ and all these boundaries are contained in $\partial_{B_0}B$.

Since $L$ is contained in the kernel of the projection $\pi_{*,1} : H_1(B) \to H_1(\widehat{B})$,
we deduce that the inclusion-induced map $H_1(\widehat{V}) \to H_1(\widehat{B})$ is surjective.

Finally, we transfer the dynamics in $B$ onto $\widehat{B}$.
Recall that $f(\partial_{B_0} B) \cap B = \emptyset$ and $f(B) \cup B \subset B'$.
Define $\widehat{f} : \widehat{B} \to \widehat{B}$ as follows.
For a point $x \in B \setminus \partial_{B_0} B$, if $f(x) \in B \setminus \partial_{B_0} B$
then $\widehat f(\pi(x)) := \pi(f(x))$ and, otherwise, $\widehat f(x) = q_{\lambda}$
provided that $\lambda = \Phi(f(x))$. Finally, $\widehat f(q_{\lambda}) := q_{\lambda'}$ if
for any (all) $x \in \lambda$, $\Phi(f(x)) = \lambda'$.
Note that $\widehat f$ is locally constant around each $q_{\lambda}$
and, more precisely, it is locally constant in $\pi(B \cap f^{-1}(B_0 \setminus B))$,
which is an open neighborhood of $Q =  \{q_{\lambda}: \lambda \in \pi_0(\partial_{B_0} B)\}$ in $\widehat B$.
It follows that $\widehat f$ is continuous.

\begin{figure}[htb]
\centering
\includegraphics[scale = 0.6]{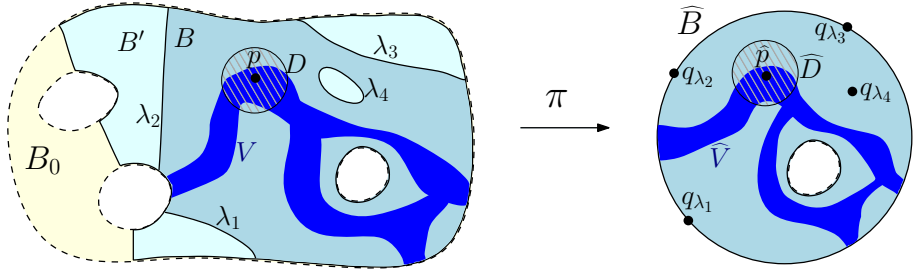}
\caption{Proof of Theorem \ref{thm:formula}.}
\label{fig:proofR2}
\end{figure}

\medskip
\emph{Step IV: Computation of the index through the map $\widehat f$.}

As $D$ does not meet $\partial_{B_0}B$, $D$ is homeomorphic to $\widehat D = \pi(D)$ and we get
$$\widehat f^l (\widehat V) \subset \widehat D.$$
Thus, we can apply Lemma \ref{lem:nilpotent} to $S = \widehat B$, $r = 1$, $W = \widehat V$, $Y = \widehat D$ and the integer $l$.
Lemma \ref{lem:nilpotent} then yields that the map $\widehat f_{*,1} : H_1(\widehat B) \to H_1(\widehat B)$
is nilpotent, so the trace of $\widehat f^n_{*,1}$ is 0 for any $n \ge 1$.

Higher homology groups are trivial except when $\widehat B$ is homeomorphic to $S^2$. In this case
$\partial B = \partial_{B'} B$, hence $f(\partial B) \cap B = \emptyset$ and $p$ is a weakly repelling fixed point.
However, since $p$ is not repelling it follows from Remark \ref{rmk:nopreimage} that there is a point $\hat{x} \in
\widehat B$ such that $\widehat{f}^{-1}(x) = \emptyset$. In particular, $\widehat f$ is not surjective and so
its degree, as a map between 2-spheres, is 0.

The previous discussion proves that the Lefschetz number of $\widehat f^n$ is equal to 1 and by Lefschetz Theorem
we conclude that
$$i(\widehat f^n, \widehat B) = 1.$$

The invariant set of $\widehat f$ is contained in the set $\{\widehat p\} \cup Q$.
Denote $\varphi$ the restriction of $\widehat f$ to $Q$.
If $q_{\lambda}$ is a fixed point for $\widehat f^n$ then $i(\widehat f^n, q_{\lambda}) = 1$ because
$\widehat f$ is locally constant at $q_{\lambda}$.
Note additionally that, since $\widehat f$ and $f$ are conjugate in a neighborhood $\widehat p$ and $p$, respectively,
we have $i(\widehat f^n, \widehat p) = i(f^n, p)$.
Therefore,
$$1 = i(\widehat f^n, \widehat B) = i(\widehat f^n, \widehat p) +
\sum_{q_{\lambda} \in Fix(\varphi^n)} i(\widehat f^n, q_{\lambda}) = i(f^n, p) + \#\Fix(\varphi^n).$$
If we denote $a_k$ the number of $k$-periodic orbits of $\varphi$
and $F = \{k \in \mathds{N} : a_k \ge 1\}$, we obtain the desired formula,
$$\{i(f^n, p)\}_{n \ge 1} =\sigma^1 - \sum_{k \in F} a_k \sigma^k.$$

\hfill \qed

\begin{remark}
In the final lines of the proof we have used that for a map $g : S \to S$ the index of a fixed point
$p \in \partial S$ is well--defined provided $S$ is an ENR. Furthermore, the index computations can be carried out
as if the map were defined in a neighborhood of $p$. We refer to \cite{marzantowicz} for a complete account on
fixed point index theory.
\end{remark}

\section{Realization of all possible sequences in Theorem \ref{thm:formula}}

In this subsection we construct examples which realize all possible fixed point index sequences
obtained in Theorem \ref{thm:formula}.

Let $F$ be an arbitrary finite subset of $\mathds{N}$ and let $a_k$ be an arbitrary positive integer for every $k \in F$.
We will define a map $f : \mathds{R}^2 \to \mathds{R}^2$ with a fixed point $p$ isolated as an invariant set and
such that
\begin{equation}\label{eq:formula}
\{i(f^n, p)\}_{n \ge 1} = \sigma^1 - \sum_{k \in F} a_k \sigma^k.
\end{equation}

Consider a continuous map $h : S^1 \to S^1$ with exactly $a_k$ attracting periodic orbits of period $k$
for every $k \in F$. Denote $P$ the set of points belonging to these orbits.
Notice that there is no obstruction in the definition of the map $h$ as we are only asking for continuity.
Take a compact neighborhood $V \subset S^1$ of $P$ such that $h(V) \subset \int(V)$ and $\Inv(h, V) = P$.

Let $g : S^1 \to \mathds{R}$ be a continuous map such that
$$g(\theta) \ge 0 \Leftrightarrow \theta \in V$$
and $g(\theta) > 0$ for $\theta \in \int(V)$. Consider the cylinder $S^1 \times \mathds{R}$
and compactify the end of the lower semi-cylinder $S^1 \times (-\infty, 0]$ with the point $\{e^-\}$.
We obtain a closed topological disk, which will be denoted $D_0$.
Similarly, we define disks $D_r$ for any $r \in \mathds{R}$.
Let $s : S^1 \times \mathds{R}\to S^1 \times \mathds{R}$ be the retraction from the cylinder onto $D_0$ defined
by $s(\theta, r) = (\theta, 0)$ if $r \ge 0$ and $s(\theta, r) = (\theta, r)$ otherwise.
Define $f : S^1 \times \mathds{R} \to S^1 \times \mathds{R}$ by
$$f((\theta, r)) = s(h(\theta), r + g(\theta)).$$
Fixing the lower end $e^-$ we obtain a continuous extension of $f$ to $S^1 \times \mathds{R} \cup \{e^-\}$.

For every $r \le 0$, the disk $D_{r}$ is an isolating block because the following property is satisfied:
$$g(\theta) \ge 0 \Rightarrow g(h(\theta)) > 0.$$
By a compactness argument, it implies that, for some $\varepsilon > 0$,
$$g(\theta) \ge - \varepsilon \Rightarrow g(h(\theta)) \ge \varepsilon.$$
Therefore, for any point $x \in D_0$ either $f^n(x)$ tends to $e^-$ as $n\rightarrow +\infty$
or $f^n(x)$ eventually belongs to $\partial D_0 = S^1 \times \{0\}$.
The same conclusion trivially holds for any point $x$ in the upper semi-cylinder.
We deduce that $\Inv(f, S^1 \times \mathds{R} \cup \{e^-\})$ is composed of a fixed point, $e^-$, and the
set $P \times \{0\}$ of attracting periodic orbits in $\partial D_0$.

The computation of the fixed point index of the point $e^-$ is now straightforward.
Take $D_1$ as the ambient space for the computation, it satisfies $f(D_1) \subset D_1$ and, clearly,
$\Inv(f, D_1) = \{e^-\} \cup (P \times \{0\})$.
The Lefschetz number of the map $f^n_{|D_1} : D_1 \to D_1$ is 1 because $D_1$ is a disk.
Consequently, the total sum of the local fixed point indices of the map $f^n_{|D_1}$ is equal to 1.
A point $\theta \in P$ is fixed by $h^n$ if and only if $(\theta, 0)$ is fixed by $f^n$. Additionally, a fixed point for the map $f^n$ of the form
$(\theta, 0)$ is attracting because $g$ is strictly positive in a neighborhood of $\theta$,
so its fixed point index is equal to 1.
Consequently, we have
$$i(f^n, e^-) = 1 - \sum_{h^n(\theta) = \theta \in P} i(f^n, (\theta, 0)) = 1 - \sum_{k | n} k \cdot a_k$$
which evidently yields Equation (\ref{eq:formula}) for $p = e^-$. \qed


\section{Non--isolated case: unbounded sequence of indices}

The purpose of this section is to show that the hypothesis of isolation as an invariant set in Theorem \ref{thm:formula}
is essential. We construct an example of a continuous map $f : \mathds{R}^2 \to \mathds{R}^2$
with a fixed point $p$ not isolated as an invariant set but isolated among the set of periodic points of $f$ and such that
the sequence $\{i(f^n, p)\}_{n \ge 1}$ is not bounded.

Before we start with the construction of $f$, let us include the following result:

\begin{theorem}[Babenko - Bogatyi \cite{babenko}, Graff - Nowak--Przygodzki \cite{graff}]
Given any integer sequence $\{I_n\}_{n \ge 1}$ satisfying Dold's congruences, there exists a map $f : \mathds{R}^2 \to
\mathds{R}^2$ with a fixed point $p$ such that $i(f^n, p) = I_n$ for every $n \ge 1$.
\end{theorem}

In the case the sequence of indices $\{I_n\}_{n \ge 1}$ is unbounded, the fixed points $p$ of the maps
provided by the theorem are accumulated by periodic orbits. This explains why our example is meaningful: the sequence
of indices will be also unbounded but the fixed point will remain isolated among the set of periodic orbits of the map.

Let us begin with our example. Firstly, define a homeomorphism in a vertical strip with a source inside.
Let $t : [-1, 1] \times \mathds{R} \to [-1, 1] \times \mathds{R}$ given by
$$(x, y) \mapsto  (x + x(1-|x|), y + s(x, y)),$$
where $s : \mathds{R} \to [-1, 1]$ satisfies $s(-1, y) = s(1, y) = 1$, $\sgn(s(0, y)) = \sgn(y)$
(here $\sgn$ denotes the sign function, in particular we obtain $s(0, 0) = 0$)
and $s$ is strictly positive outside the substrip $\{|x| \le 1/2\}$.
Assume further that $s$ is chosen so that $t$ is a homeomorphism.

It follows from the definition that $(0,0)$ is the unique fixed point of $t$ and that the forward orbit
of any point $(0,y)$, with $y < 0$, tends to the lower end. As the forward orbit of any point
outside the line $x = 0$ eventually exits the substrip $\{|x| \le 1/2\}$, they all converge
to the upper end. Finally, notice that $(0,0)$ is a source of the homeomorphism
$t$, so that for every $n \ge 1$
\begin{equation}\label{eq:repelling}
i(t^n, (0,0)) = 1.
\end{equation}

Let us proceeded to the example. The construction will be done within the same
setting as in the previous section:
identifying the plane $\mathds{R}^2$ with the compactification of $S^1 \times \mathds{R}$ with the lower end.
Consider a degree--2 covering map $h : S^1 \to S^1$ which pointwise fixes an interval $I_0$ centered at
the angle $\theta = 0$, i.e. $h(x) = x$ for every $x \in I_0$.
Define $f : S^1 \times \mathds{R} \to S^1 \times \mathds{R}$ by $f(\theta, r) = (h(\theta), r+1)$.

As usual, we work in $S^1 \times \mathds{R} \cup \{e^-\}$, where $e^-$ denotes the lower end of the cylinder.
The map $f$ is naturally extended to $e^-$ making it a fixed point.
The strip $A = \{(\theta, r): \theta \in I_0\}$ satisfies $f(A) = A$.
Identify $A$ to $[-1,1] \times \mathds{R}$ via a linear transformation $g$ which fixes the second coordinate.
Redefine the map $f$ inside $A$ as $f_{|A} = g^{-1} \circ t \circ g$.
The point $q = (0,0)$ has now become a source for $f$. Moreover, the forward
orbit of every point in $A$ goes to the upper end of the cylinder, denoted $e^+$, except for the points of the form $(0, r)$ with $r < 0$ whose forward orbit converge to $e^-$.

If we now extend the map $f$ also to the upper end by $f(e^+) = e^+$, we obtain a map conjugate to a degree--2 map
in a 2-sphere. Since $e^+$ is then an attracting fixed point, it follows that $i(f^n, e^+) = 1$
for every $n \ge 1$. Lefschetz theorem implies that
$$i(f^n, e^-) + i(f^n, q) = 2^n,$$
and since $f$ is locally conjugate to $t$ around the fixed points $q$ and $(0,0)$ respectively,
we deduce from (\ref{eq:repelling}) that
$$i(f^n, e^-) = 2^n - 1$$
for every $n \ge 1$.

\begin{figure}[htb]
\centering
\includegraphics[scale = 0.4]{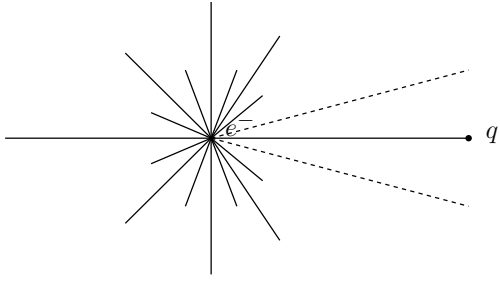}
\caption{$\Inv(f, D)$ sketched in the plane $S^1 \times \mathds{R} \cup \{e^-\}$, where the
center of the invariant set correspond to the fixed point $e^-$ and the dashed lines represent
the border of the strip $A$.}
\label{fig:cantorejemplo}
\end{figure}

Finally, let us describe the non--trivial invariant set for $f$ which contains $e^-$, see Figure \ref{fig:cantorejemplo}.
Outside the strip $A$, the radial coordinate is increased by $f$. Moreover, the positive orbit
of a point in $A$ either tends to $e^+$ or belongs to the half-line $R_0 = \{(0, r): r \le 0\}$
and so converges to $e^-$. Thus, the only periodic point of $f$ other than the ends is the point $q = (0,0)$.
Furthermore, the maximal invariant subset of $f$ inside the topological disk $D = S^1 \times (-\infty, 0] \cup \{e^-\}$
is exactly the backward orbit of $R_0$,
$$\Inv(f, D) = \bigcup_{n \ge 0}f^{-1}(R_0),$$
which is composed of the countably union of half lines of the form
$$\{(\theta_0, r): h^k({\theta_0}) = 0, r \le k\}$$
for some non--positive integer $k$.
The dynamics within $\Inv(f, D)$ is easy: the forward orbit of every point tends to $e^-$ except
for the end points of the half lines which are eventually mapped onto $q$.

In sum, when one ask the fixed points only to be non--accumulated by other periodic points
we may have an unbounded fixed point sequence.
However, this example is not completely satisfactory as its behavior resembles very much
to a source. Incidentally, notice that the restriction of the map to the cylinder is a 2:1 covering map.

\textbf{Question:} Is it possible to construct an example in which the behavior of $\{i(f^n, p)\}_{n \ge 1}$
differs to a great extent from a geometric progression?

\section{Lefschetz zeta function}

The Lefschetz zeta function associated to the fixed point index sequence $I = \{I_n\}_{n \ge 1}$ is defined as
$$\mathcal Z_f(t) = \exp\left(\sum_{n = 1}^{\infty}\frac{I_n t^n}{n}\right)$$
It is easy to see that if $I = \sigma^k$ then $\mathcal{Z}_f(t) = 1/(1 - t^k)$.
Consequently, if $I = \sum_{k \ge 1} a_k \sigma^k$
we obtain the formal identity $\mathcal Z_f (t) = \prod_{k \ge 1} (1 - t^k)^{-a_k}$.
This product is finite if and only if any (hence all) statement in Lemma \ref{lem:index} holds.
When $I$ is the fixed point index sequence at a fixed point $p$ we denote the (local) Lefschetz
zeta function $\mathcal{Z}_f^p$. For a more detailed account on this object see \cite{franksbook}
and also \cite{babenko}.

The Lefschetz zeta function is a bookkeeping device which contains the same information as the fixed point
index sequence in a more condensed way. In this language, Theorem \ref{thm:formula} takes the following form:

\begin{corollary}
Let $U$ be an open subset of the plane, $f : U \to f(U) \subset \mathds{R}^2$ continuous and $p$ a fixed point for $f$
which is isolated as an invariant set and neither a sink nor a source. Then
\begin{equation}\label{eq:zeta}
\mathcal Z_f^p(t) = \frac{(1 - t^{k_1})^{b_1}(1 - t^{k_2})^{b_2} \ldots (1 - t^{k_l})^{b_l}}{1 - t}
\end{equation}
for some integer $l \ge 1$ and positive integers $b_1, \ldots, b_l$.
\end{corollary}

Theorem \ref{thm:formula} allows to extend the work of Franks in \cite{frankszeta}. Define the degree
of a source as the integer $i(f, p)$.
\begin{proposition}
Let $f$ be a continuous self-mapping of $D^2$ or $S^2$ such that every periodic orbit of $f$ is isolated as an invariant set.
Assume further that the set of periodic points of $f$ is finite.
Denote $\mathcal A$ the set of orbits of sinks and degree--1 sources of $f$ and $\mathcal S$ the set of orbits of sources
with degree different from -1, 0, 1.
Then
\begin{enumerate}
\item $\# \mathcal S \le 1$ and $\# \mathcal S = 1$ only in the case $f : S^2 \to S^2$ has degree $d$,
with $|d| \ge 1$, and the only element in $\mathcal S$ is a degree--$d$ source and it is fixed.
\item $\# \mathcal A \ge 1$ and the inequality is strict for degree-1 maps $f$ of $S^2$.
\item Given a degree--(-1) source of period $m$
	\begin{enumerate}
	\item either there exists another periodic point of period $m$
	\item or $m$ is even and
there is at least one degree--(-1) source of period $m/2$.
	\end{enumerate}
\end{enumerate}
\end{proposition}

The proof only involves minor modifications of Franks' argument. It is based on the following simple fact (assume we have
a finite number of periodic orbits): the fixed point index
of $f$ is the sum of the fixed point indices of every periodic point, which translates to
\begin{equation}\label{eq:zetaproduct}
\mathcal Z_f(t) = \prod \mathcal Z_f^p(t),
\end{equation}
where the product runs over all periodic orbits of $f$.

On the one hand, we have that the fixed point indices of a degree-$d$ map of $S^2$ are $i(f^n) = 1 + d^n$, hence
\begin{equation*}
\mathcal Z_f(t) = \frac{1}{(1-t)(1-dt)}.
\end{equation*}
The zeta function of a self-map of the disk $D^2$ is $\mathcal Z_f(t) = 1/(1-t)$.

On the other hand, the local contribution of the orbit of each periodic point $p$ of period $m$ is:
\begin{itemize}
\item If $p$ is a sink, $\mathcal Z_f^p(t) = 1/(1-t^m)$.
\item If $p$ is a source of degree $r$, $\mathcal Z_f^p(t) = 1/(1-rt^m)$.
\item Otherwise, replace $t$ by $t^m$ in Equation (\ref{eq:zeta}).
\end{itemize}

\begin{proof}
Denote $\mathcal S'$ the set of orbits of periodic sources of degree--(-1) and $\mathcal H$ the
set of orbits periodic points which are neither sinks nor sources.
The case of self-maps of $D^2$ is for our purposes the same as of a degree--0 self-map of $S^2$.
Assume $f: S^2 \to S^2$ has degree $d$.

After removing denominators in Equation \ref{eq:zetaproduct} we obtain
\begin{equation*}
(1-t)(1-dt) \prod_{p \in \mathcal H} \prod_{i=i(p)} (1-t^{k_im_p})^{b_i} =
\prod_{p \in \mathcal H \cup \mathcal A} (1-t^{m_p})
\prod_{p \in \mathcal S'} (1 + t^{m_p}) \prod_{p \in \mathcal S} (1 - r_p t^{m_p})
\end{equation*}
In the products we only take one point $p$ representing each periodic orbit and $m_p$ denotes the period of the orbit.
Firstly, look at the pieces of both polynomials not located in the unit circle. Clearly, one must have
that $(1-dt)$, if $|d|\ge 1$, or 1, otherwise, is equal to a product of factors of type $(1 - d_p t^{m_p})$.
Evidently, there must be only one of such factors and correspond to the orbit of a source of period 1 and degree $d$ in the case
in which $|d| \ge 1$ and none otherwise. This proves (1).

In order to prove (2), it suffices to compute the multiplicity of 1 as a root for both polynomials.
In the RHS its multiplicity is equal to $\# \mathcal H + \# \mathcal A$ whereas in the LHS it is greater
or equal than $1 + \# \mathcal H$ if $d \neq 1$ or $2 + \# \mathcal H$ if $d = 1$.

For (3), in view that in the LHS remain only factors of type $(1-t^s)$
the only way to cancel out a factor of the form $(1+t^m)$ in the RHS is via a factor $(1-t^m)$, which may arise
from a $m$--periodic orbit or from the product $(1-t^{m/2})(1+t^{m/2})$ if $m$ is even.
\end{proof}

Proposition 20 can be rewritten to express sufficient conditions for the existence of infinite periodic orbits
in a class of maps, those for which every periodic orbit is isolated as an invariant set, which contains
that of $C^1$ maps for which every periodic point is hyperbolic.

\begin{corollary}
Using the notation of Proposition 20, if any of the following conditions is fulfilled then $f$ has
infinitely many periodic orbits.
\begin{itemize}
\item $\#\mathcal S \ge 2$.
\item The unique element of $\mathcal S$ is the orbit of a degree--$r$ source which is not a fixed point or
the map $f : S^2 \to S^2$ has degree different from $r$.
\item $\mathcal A$ is empty or contains just one periodic orbit and $f : S^2 \to S^2$ has degree 1.
\item There is a degree--$(-1)$ source with odd period $m$ and no other $m$--periodic point.
\item There is a degree--$(-1)$ source with even period $m$ and no degree--$(-1)$ source with period $m/2$.
\end{itemize}
\end{corollary}

In the presence of infinite periodic orbits we can bound from below the number of periodic orbits of $f$
of each period.

\begin{proposition}
Let \(f: S^2 \rightarrow S^2\) be a degree--\(d\) map such that all  its periodic orbits are isolated as invariant sets.
Then, if \( f\) has no sources of degree \(r\) with \(|r|>1\) (this happens, in particular, if $f$ is $C^1$) we have that \( N_n(f) >d^n\).
\end{proposition}

\begin{proof}
From Corollary \ref{cor:index<1} and Proposition \ref{prop:sources} we get that the index of any fixed point of $f^n$ is
bounded by 1. Therefore,
$$N_n(f) \ge \sum_{p \in N_n(f)} i(f^n, p) = 1 + d^n$$
by Lefschetz Theorem.


\end{proof}

The previous proposition depicts a special case in which the Growth Rate Inequality (Equation \ref{eq:growth}) holds.

\end{document}